\documentclass[11pt]{article}
\usepackage{stmaryrd}
\usepackage{amsfonts}
\usepackage{amsmath, amsthm}
\usepackage{amssymb}
\usepackage{eucal}
\newtheorem{theorem}{Theorem}[section]

\begin{document}

\baselineskip 16pt
\title{On finite groups with elements of prime power orders\thanks{This paper is completed in July 1981, under
the supervision by Professor Zhongmu Chen in Southwest Normal
University. Contact address: shiwujie@outlook.com }}

\author{{Wujie Shi, \ \ Wenze Yang}\\
{\small Southwest Normal University, Chongqing, China}}
\date{}
\maketitle
\begin{abstract} This paper is published in Journal of Yunnan Education College, no.1(1986), p.2-10. (Internal Version) and written in Chinese. Translate it to English is helpful for the citing some conclusions in this paper.
\end{abstract}

\noindent All groups considered are finite.

G. Higman in \cite{Hig} studied the groups in which every element
has prime power order.  For convenience, we call them  EPPO-groups.
Later, Feit, Hall and Thompson et al investigated  the structure of
$CN$-groups, which is more extensive than EPPO-groups and firstly
concluded that the order of a non-solvable $CN$-group is even (see
\cite{FHT}). Suzuki in \cite{Su} studied the class of $CIT$-groups,
which is wider than the class of EPPO-groups, and in which the order
of the centralizer of an element with order 2 is a $2$-group.
Suzuki's paper \cite{Su} is very profound, which almost classified
the non-solvable $CIT$-groups and proved the equivalence of
non-solvable  $CN$-groups and non-solvable $CIT$-groups. Since the
order of a $CIT$-group is even, the EEPO-groups, in particular the
solvable EPPO-groups of odd orders, are not contained the class of
$CIT$-groups.

In this paper, we continue discussing the structure of groups in
which every element has prime power order and independently obtain
some more detailed results than Higman's results in \cite{Hig}. In
addition, by Suzuki's results in \cite{Su}, we determine the types
of non-solvable EPPO-groups, which imply an interesting result, that
is, the smallest non-abelian simple group $A_5$ could be
characterized only by its two `orders', that is, the characteristic property of $A_5$ is: (1) The order of group contains at least three prime factors, and (2) the order of every nonidentity element is prime.

Throughout this paper, we denote by $\pi(G)$ the set of all primes
dividing the order of a group $G$ and denote by $G_p$ a Sylow
$p$-subgroup of $G$ for some $p\in \pi(G)$.

\section{General Porperties}

The following observation is clear.
\begin{theorem}
The subgroups and quotients of an EPPO-group are also EPPO-groups.
\end{theorem}

\begin{theorem}
Let $G$ be a group. then the following statements are equivalent:

$\mathrm{(1)}$ $G$ is an EPPO-group.

$\mathrm{(2)}$ For any non-trivial $x$  and $y$ in $G$ with $(|x|,
|y|)=1$, $xy\neq yx$.

$\mathrm{(3)}$ For any non-trivial $x$  and $y$ in $G$ with $(|x|,
|y|)=1$, $C_G(x)\cap C_G(y)=1$.

$\mathrm{(4)}$ For any non-trivial $p$-subgroup $A\leq G$,
$C_G(A)\leq P$, where $p$ ia a prime dividing the order of $G$ and
$P$ is a Sylow $p$-subgroup of $G$.
\end{theorem}

\begin{proof}
(1)$\Rightarrow$(2): Let $a, b\in G$ such that $|a|=p^{\alpha}$ and
$|b|=q^{\beta}$, where $p\neq q$. If $ab=ba$, then
$|ab|=p^{\alpha}q^{\beta}$, a contradiction because $G$ is an
EPPO-group.

(2)$\Rightarrow$(3): Let $a, b\in G$ such that $(|a|, |b|)=1$.
Suppose that $C_G(a)\cap C_G(b)\neq 1$. Then for some $1\neq g\in
C_G(a)\cap C_G(b)$, we have that either $(|g|, |a|)=1$ or $(|g|,
|b|)=1$ by the hypothesis. But since $g$ centralizes both $a$ and
$b$, this yields a contradiction.

(3)$\Rightarrow$(4): Note that $$C_G(A)=\bigcap_{y\in A}C_G(y).$$ If
$A$ is a $p$-group, then for any nontrivial element $y\in A$,
$C_G(y)$ is also a $p$-group by the assumption. Thus, $C_G(A)$ is a
$p$-group by the observation above and consequently there exists a
Sylow $p$-subgroup $P$ of $G$ such that $C_G(A)\leq P$.

Next, we suppose that there exists at least two different primes $p$
and $q$ dividing the order of $A$. Let $a$ be an element of order
$p$ in $A$ and $b$ be an element with order $q$ in $A$. Then
$C_G(a)\cap C_G(b)=1$ by the hypothesis and so $C_G(A)=\bigcap_{y\in
A}C_G(y)=1$, as desired.

(4)$\Rightarrow$(1): If $G$ is not an EPPO-group, then we may assume
that  $G$ contains an element $g$ of order $p^{\alpha}q^{\beta}$
with $\alpha\neq 0$ and $\beta\neq 0$. It is clear that $g\leq
C_G(\langle g\rangle)\neq 1$, a contradiction by the assumption of
(4).
\end{proof}

\begin{theorem}
Let $G$ be a non-cyclic metacyclic $\mathrm{EPPO}$-group such that
$|\pi(G)|\geq 2$. Then $G=\langle a, b\rangle$ such that
$a^{p^{\alpha}}=1$, $b^{q^{\beta}}=1$, $b^{-1}ab=a^r$, where $p$ and
$q$ are different primes such that  the exponent of $r (mod \ p^{\alpha})$  is  ${q^{\beta}}$.
\end{theorem}

\begin{proof}
We first prove the necessity.

Let $G$ be a non-cyclic metacyclic EPPO-group. Then $G'$ and $G/G'$
are cyclic and moreover are EPPO-groups by the hypothesis. Therefore
$|G'|=p^{\alpha}$ and $|G/G'|=q^{\beta}$ for some $p, q\in \pi(G)$.
Since $|\pi(G)|\geq 2$, we see that $p\neq q$ and so $G$ is of order
$p^{\alpha}q^{\beta}$ such that every Sylow subgroup of $G$ is
cyclic. By \cite[Theorem 9.4.3]{Hall}, we have that

\begin{gather*}\label{$*$}
G=\langle a, b|a^{p^{\alpha}}=1, b^{q^{\beta}}=1, p\neq q,
b^{-1}ab=a^r, (p, r-1)=1,  r^{q^{\beta}}\equiv 1\pmod{p^{\alpha}}\rangle. \tag{$*$} \\
\end{gather*}

We claim that $p^{\alpha}\mid (r^{q^{\beta}}-1)$. If not, suppose
that for some $0\leq s< \beta$, we have $p^{\alpha}|(r^{q^{s}}-1)$.
Then $b^{-q^{s}}ab^{q^s}=a^{r^s}=a$, which implies that
$ab^{q^s}=b^{q^s}a$, contradicting that $G$ is an EPPO-group. Hence
our claim holds. It follows that $p\nmid r-1$. In fact, if $p\mid
(r-1)$, then we can write $r=kp+1$. Then
$r^{p^{\alpha-1}}=(kp+1)^{p^{\alpha-1}}$, form which we conclude
that $r^{p^{\alpha-1}}\equiv 1\pmod{p^{\alpha}}$. By the argument
above, we get that $q^{\beta}|p^{\alpha-1}$, a contradiction because
$p\neq q$. Thus the necessity of this theorem is proved.

Now we prove the sufficiency. By \cite[Theorem 9.4.3]{Hall}, the
Sylow subgroups of a group $G$ satisfying relation (\ref{$*$}) are
all cyclic and so $G$ is a metacyclic group such that $G=\langle
a\rangle\rtimes \langle b\rangle$. Let $g=b^xa^y$ be any element of
$G$, where $x$ and $y$ are natural numbers. If $x=0$, then $g$ is an
 element of prime power order. Suppose that $x\neq 0$. Then
 $$(b^xa^y)^{q^{\beta}}=b^{xq^{\beta}}a^{(r^{q^{\beta-1}}+r^{q^{\beta-2}}+\cdots +1)y}.$$
Since $$p^{\alpha}\mid
r^{q^{\beta}}-1=(r-1)(r^{q^{\beta-1}}+r^{q^{\beta-2}}+\cdots +1),$$
and $p\nmid r-1$, we have $$p^{\alpha}\mid
(r^{q^{\beta-1}}+r^{q^{\beta-2}}+\cdots +1),$$ which implies that
$(b^xa^y)^{q^{\beta}}=1$. Thus, $G$ ia an EPPO-group and therefore
$G$ is noncyclic.
\end{proof}


\begin{theorem}
Let $G$ be an $\mathrm{EPPO}$-group and $H$ be a subgroup of $G$
such that $(|H|, d)=1$ for a natural number $d>1$. Then $|H|$
divides the number of elements of order $d$ in $G$.
\end{theorem}

\begin{proof}
If $d$ does not divide the order of $G$, then the result is trivial.
Hence we may assume that $d||G|$. let $\Omega$ denote the set of
elements of order $d$ in $G$. Take an element $a\in \Omega$ and let
$C_1$ denote the orbit of $a$ by the conjugate action of $H$ on
$\Omega$. By Theorem 1.2, we have that $N_H(a)=C_H(a)=1$ and so
$|C_1|=|H:C_H(a)|=|H|$ by \cite[Theorem 1.6.1]{Hall}. Let $b\in
\Omega\setminus C_1$. Then we have other $H$-orbit $C_2$ of size
$|H|$. Continuing like this, we have that
$\Omega=\bigcup_{i=1}^tC_i$, where each $C_i$ is an $H$-orbit of
size $H$. It follows that $|H|\mid |\Omega|$, implying the result.
\end{proof}

\noindent{\bf Remark 1.}  After 30 years, we find that Theorem 1.4 is also the characteristic
property of an $\mathrm{EPPO}$-group. See [A.A. Buturlakin, Rulin Shen and Wujie Shi, Siberian Math. J., 58, no.3(2017), 405-407.].

\noindent {\bf Corollary 1.1.} Let $G$ be an EPPO-group and $N$ be a
normal subgroup of $G$. Then $$\prod p^{\alpha}\mid (|N|-1),$$ where
$p^{\alpha}$ is the $p$-part of $|G|$ and $p\notin \pi(N)$.

\section{Solvable EPPO-groups}

\noindent {\bf Lemma 2.1.} \emph{Let $G$ be a non-abelian solvable
EPPO-group. Then $G$ has a normal abelian subgroup which is not
contained the center of $G$.}

\noindent {\bf Theorem  2.1.} \emph{A solvable EPPO-group $G$ is an
$M$-group. That is, the every irreducible representation of $G$ is a monomial representation}
\begin{proof}
This proof is similar to [5, Theorem 16].
\end{proof}

\noindent {\bf Lemma 2.2.} \emph{Let $N$ be an elementary abelian
group of order $q^m$ and $H$ be an EPPO-group such that $(|H|,
|N|)=1$. Suppose that $H$ acts on $N$ and denote by $G=N\rtimes H$
the semidirect product of $N$ by $H$. Then $G$ is an EPPO-group if
and only if $H$ acts faithfully on $N$, and for any non-trivial
element $h\in H$, $1$ is not an   eigenvalue of $h$ as a linear
transformation on $N$.}

\begin{proof}
We first prove the necessity. Suppose that $G$ is an EPPO-group. Let
$h$ and $a$ be non-trivial elements in $H$ and $N$ respectively.
Then $a^h=h^{-1}ah$ by the definition of semidirect. If $h$ has a
characteristic root $1$, then there exists an element $1\neq a_1\in
N$ such that $a_1^h=1\cdot a_1=a_1$ and therefore $a_1h=ha_1$, a
contradiction. In addition, it is easy to see that $H$ acts
faithfully on $N$.

Now we prove the sufficiency.

 By the hypothesis, $H$ is an EPPO-group and so $|h|=p^{\alpha}$ for some
$p\in \pi(H)$. Observing that $h$ is a linear transformation on $N$,
$h$ is a root of the polynomial $\lambda^{p^{\alpha}}-1$. Assume
that $h\neq 1$. Since $1$ is not an   eigenvalue of $h$ as a linear
transformation on $N$, the minimal polynomial of $h$  is
$\lambda^{p^{\alpha}-1}+\lambda^{p^{\alpha}-2}+\cdots+1$. Let $ha$
be any element in $G$ with $h\in H$ and $a\in N$. If $h=1$, then
$(ha)^q=1$. If $h\neq 1$ and $h^{p^{\alpha}}=1$, then
$$(ha)^{p^{\alpha}}=h^{p^{\alpha}}\cdot a^{[h^{p^{\alpha}-1}+h^{p^{\alpha}-2}+\cdots+1]}=1\cdot a^0=1,$$
which shows that $G$ is an EPPO-group.
\end{proof}

\noindent {\bf Lemma 2.3.} \emph{Let $M$ be an $n\times n$ monomial
matrix and $M^p=I_n$ for a prime $p$. If $M$ is not a diagonal
matrix, then $1$ is an eigenvalue of $M$.}

\begin{proof}
Let $M$ be an $n\times n$ monomial matrix. Since $M^p=I_n$, we have
that $|M|\neq 0$ and so in each row and column of $M$ there is
exactly one non-zero element. Denote by $a_{i\sigma(i)}$ the
non-zero element of $M$ in the \emph{i}th row, where $\sigma$ is a
permutation on the set $\{1, 2, \cdots, n\}$ and the order of
$\sigma$ is $p$ by the hypothesis. Therefore $\sigma=\sigma_1
\sigma_2\cdots \sigma_s$, where every $\sigma_i$ is a cyclic
permutation of order $p$. Since $M$ is not a diagonal matrix, there
exists at least one non-trivial $\sigma_i$ and without loss of
generality, one can assume that $\sigma_1=(12\cdots p)$ and so

\begin{equation*}
M=\left(
  \begin{array}{cc}
    C & 0\\
    0 & *\\
  \end{array}
\right),
\end{equation*}
where
\begin{equation*}
C=
 \left(
  \begin{array}{cccccc}
    0 & a_{12} & 0 & \cdots & 0 & 0\\
    0 & 0      & a_{23} & \cdots &0  & 0 \\
    \vdots & \vdots & \ddots\\
    0 & 0 & 0 & \cdots & 0 & a_{p-1, p}\\
    a_{p1}
  \end{array}
\right).
\end{equation*}
Since $M^p=I_n$, we have that $C^p=I_n$, which induces that
$a_{12}a_{23}\cdots a_{p-1,p}a_{p1}=1$. Thus

\begin{align*}
|I-C| &=
 \left |
  \begin{array}{cccccc}
    1 & -a_{12} & & &  & \\
     &     1 & -a_{23} &  &  &  \\
    \vdots & \vdots & \ddots\\
     &  &  &  & 1 & -a_{p-1, p}\\
    -a_{p1}
  \end{array}
\right |\\
&=1-a_{12}a_{23}\cdots a_{p-1,p}a_{p1}=0.
\end{align*}
It follows that $1$ is an eigenvalue of $C$ and is also an
eigenvalue of $M$.
\end{proof}

Combining Lemma 2.2 and Lemma 2.3, we get the following result.

\noindent {\bf Theorem  2.2.} \emph{Let $G$ be an EPPO-group with a
non-trivial normal $q$-subgroup $Q$ for a prime $q$. Let $H$ be a
solvable subgroup of $G$ such that $(|H|, q)=1$. Then $H$ is either
a cyclic $p$-group or a generalized quaternion group. In particular,
if $p\neq q$, the Sylow $p$-subgroups of $G$ are cyclic or
generalized quaternion and furthermore, if $G$ is solvable, the
order of $G$ is $p^{\alpha}q^{\beta}$ for some positive integers
$\alpha$ and $\beta$.}

\begin{proof}
Let $N$ be a minimal $H$-invariant subgroup contained in $Q$. Then
$N$ is an elementary $q$-group. Since $G$ is an EPPO-group, we have
that $H$ is also an EPPO-group by Theorem 1.1. It follows that the
representation of $H$ on $N$ is a monomial representation. Let $h\in
H$ and let $h$ be of prime order. We claim that the representation
matrix of  $h$ is diagonal. If not, by Lemmas 2.2 and 2.3, we obtain
that $G$ is not an EPPO-group, contradicting our assumption on $G$.
Thus, if $h$ and $k$ are two elements in $H$ of prime orders, then
$hk=kh$. Hence $H$ must be a $p$-group some prime $p$ since $H$ is
an EPPO-group. Furthermore, by \cite[Theorem 7.24]{Kur}, we conclude
that $H$ is a cyclic group or a generalized quaternion group.

Suppose that $G$ is solvable. Then, for prime $q$,  $G$ has a Hall
$q'$-subgroup $H$ and consequently by the foregoing arguments, $H$
must be a $p$-group for some $p\in \pi(G)$. Therefore
$|G|=p^{\alpha}q^{\beta}$.
\end{proof}

\noindent {\bf Theorem  2.3.} \emph{Let $G$ be an EPPO-group. Then
the following statements hold.}

(1) \emph{If $G$ has a generalized quaternion Sylow $2$-subgroup,
then $G$ is of order $2^{\alpha}q^{\beta}$ with $\beta\geq 0$ and
the Sylow $q$-subgroup of $G$ is normal in $G$.}

(2) \emph{If $G$ has a non-trivial normal $q$-subgroup $Q$ for a prime $q$,
then $G$ is solvable provided one of the following holds.}

(i) \emph{$q$ is an odd  prime.}

(ii) \emph{$Q$ is the Sylow $q$-subgroup of $G$. }

(iii)\emph{ $G$ has an abelian Sylow $2$-subgroup.}

\begin{proof}
(1) By the hypothesis and \cite{BS}, we have that there exists a
normal subgroup of $N$ of odd order such that $G/N$ has a central
element of order $2$. Since $G/N$ is an EPPO-group, $G$ must be a
$2$-group. By the oddness of $N$, we know that $G$ is solvable. By
Theorem 2.2, the order of $G$ is $2^{\alpha}q^{\beta}$ and $N$ is a
normal Sylow $q$-subgroup in $G$.

(2) Let $|G|=p_1^{\alpha_1}p_2^{\alpha_2}\cdots
p_s^{\alpha_s}q^{\beta}$. By (1), we may assume that the Sylow
$2$-subgroups of $G$ are not generalized quaternion groups.

(i) Suppose that $q$ is an odd prime. If $p_i\neq 2$ for $1\leq
i\leq s$, then $G$ is of odd order and so $G$ is solvable. If
$p_i=2$ for some $i$, then by Theorem 2.2, the Sylow $2$-subgroups
of $G$ are cyclic. It follows from Burnside's Theorem (see
\cite[Theorem 14.3.1]{Hall}) that $G$ has a normal $2$-complement
and so $G$ is solvable.

(ii) If $Q$ is a Sylow $q$-subgroup of $G$, then $Q$ is a normal
Sylow $q$-subgroup of $G$. By Theorem 2.2, we have that every Sylow
$p$-sbgroup of $G$ is cyclic for $p\neq q$. Thus, $G/Q$ is a
metacyclic group and therefore $G$ is solvable.

(iii) Let $P$ be a Sylow $2$-subgroup of $G$ and suppose that $P$ is
abelian. If $q\neq 2$, then the result follows from case (i). If
$q=2$, then $C_G(Q)$ is a Sylow $2$-subgroup of $G$ by the
hypothesis on $G$. Since $C_G(Q)$ is normal in $G$, we obtain that
$G$ is solvable.

Thus, the proof is complete.
\end{proof}

\noindent {\bf Lemma 2.4.} \emph{Let $G$ be a group of order
$p^{\alpha}q^{\beta}$, $P$ be a Sylow $p$-subgroup of $G$ and $Q$ be
a Sylow $q$-subgroup of $G$. Suppose that $P$ is cyclic and $Q$ is
minimal normal in $G$. If $C_G(Q)=Q$, then $\beta$ is the exponent of $q (mod \ p^{\alpha})$.
Especially, if $G$ is a EPPO-group, then $C_G(Q) = Q$ is satisfied.}

\begin{proof}
Since $Q$ is an elementary abelian group of order $q^{\beta}$ and
$G=PQ$, we have that $Q$ is also a minimal $P$-invariant subgroup by
the minimality of $Q$. Then $P$ acts irreducibly and faithfully on
$Q$. Set $P=\langle a\rangle$. Suppose that under some basis of $Q$,
the representation matrix of $a$ is $A$. By \cite[Ch.3, Theorem
2]{Jac}, we know that the characteristic  polynomial of $A$ is equal
to its minimum polynomial. Since the order of $a$ is $p^{\alpha}$,
we have that $A^{p^{\alpha}}=I_{\beta}$, which implies that the
characteristic polynomial $f(x)$ of $A$ divides $x^{p^{\alpha}}-1$,
where $f(x)\in \mathbb{F}_q[x]$ and $\mathbb{F}_q$ denotes the
finite field with $q$ elements. Let $k(q^{\beta})$ denote the
splitting field of $f(x)$ over $\mathbb{F}_q$. Then all
characteristic roots of $f(x)$ are $p^{\alpha}$th roots of unity in
$k(q^{\beta})$ and $f(x)$ has no multiple root in $k(q^{\beta})$. We
claim that $f(x)$ is irreducible over $\mathbb{F}_q$. If not,
suppose that $f(x)=f_1(x)f_2(x)$, then it is easy to see that $A$ is
similar to the following block matrix
\begin{equation*}
\left(
  \begin{array}{cc}
    A_1 & \\
        & A_2\\
  \end{array}
\right),
\end{equation*}
where $f_i(x)$ is the characteristic polynomial for $A_i$,
contradicting the irreducibility of $A$. It follows that $f(x)$ is
irreducible over $\mathbb{F}_q$.

Let $\omega$ be a root of $f(x)$ in $k(q^{\beta})$. Then the number
of conjugate elements with $\omega$ is $\beta$. Let $\zeta$ generate
$k(q^{\beta})$. Then the map $$\phi: \ \zeta \rightarrow \zeta^q$$
generates the automorphism group of $k(q^{\beta})$. It is clear that
$\phi(\omega)=\omega^q$. Since the order of $A$ is $p^{\alpha}$, we
get that the order of $\omega$ is also $p^{\alpha}$. This implies
that the conjugacy class containing $\omega$ consists of
$\omega^{q^0}, \omega^{q^1}, \ldots, \omega^{q^r-1}$ where
$q^r\equiv 1 (\mathrm{mod}  \ p^{\alpha})$. Thus $\beta=r$, which
completes the proof.

\end{proof}

\noindent {\bf Theorem  2.4.} \emph{Let $G$ be a solvable EPPO-group
such that $|\pi(G)|>1$ and $Q$ be a maximal normal $q$-subgroup of
$G$ for some $q\in \pi(G)$. Then}

(1) \emph{Suppose  that  $G_2$ is not a generalized quaternion group,
where $G_2$ is a Sylow $2$-subgroup of $G$. Then $G/Q$ is a
meta-cyclic  group. Let $|G|=p^{\alpha}q^{\beta}$. Then $G/Q$ is of
order $p^{\alpha}q^{\gamma}$ with $q^{\gamma}|(p-1)$ and the chief
series of $G$ is as follows:
$$\underbrace{q, \ldots, q}_{\gamma}; \underbrace{p, \ldots, p}_{\alpha}; q^{b_1}, \ldots, q^{b_k}; \ \ b|b_i, i=1, \ldots, k, \ \ \gamma<b,$$
where $p^{\alpha}|(q^b-1)$. If $Q$ is the Sylow $q$-subgroup of $G$,
then the chief series of $G$ is as follows:
$$\underbrace{p, \ldots, p}_{\alpha}; \underbrace{q^{b}, \ldots, q^b}_{k}, \ \ \beta=kb,$$
and the length of nilpotency class of $Q$ is bounded by $k$.}

(2) \emph{If the Sylow 2-subgroups of $G$ are generalized quaternion
groups, then $G$ has the following chief series
$$\underbrace{2, \ldots, 2}_{\alpha}; q^{b_1}, \ldots, q^{b_k},$$
where $b_i>1$, $b|b_i$ for $i=1, \ldots, k$, where  $b$ is the exponent of $q (mod \ 2^{\alpha-1})$}.

\begin{proof}
(1) Since $G$ is an EPPO-group, we have that
$|G|=p^{\alpha}q^{\beta}$, where $p, q\in \pi(G)$ and $p\neq q$. If
$Q$ is a Sylow $q$-subgroup of $G$, then $G/Q$ is of order
$p^{\alpha}$ and by Theorem 2.2, $G/Q$ is a cyclic group. If
$|Q|<q^{\beta}$, then $G/Q$ has a cyclic Sylow $p$-subgroup since
the Sylow $p$-subgroups of $G$ are cyclic. By our hypothesis, $G/Q$
has no normal Sylow $q$-subgroup, that is $O_q(G/Q)=1$. It follows
that there exists a normal $p$-subgroup of order $p$ in $G/Q$. Since
the automorphism group of a group of order $p$ is cyclic, we obtain
that the Sylow $q$-subgroups of  $G/Q$ are cyclic and therefore
$G/Q$ is a metacyclic group.

We first assume that $Q$ is a Sylow $q$-subgroup of $G$. By the
foregoing arguments, we see that $G$ has  a chief series
$$G=C_0>C_1>\cdots >C_{\alpha-1}>Q>C_{\alpha+1}>\cdots >C_{s-1}>C_s=1.$$
Suppose that $P$ is a Sylow $p$-subgroup of $G$. Set $G_1=PC_{s-1}$.
Then $C_{s-1}$ is a minimal normal subgroup of $G_1$. Otherwise, let
$1\neq C_t<C_{s-1}$ such that $C_t$ is normal in $G_1$. Since $Q$ is
normal in $G$, we get that $Z(Q)$ is normal in $G$.  Since $Z(Q)\cap
C_{s-1}\neq 1$, we have that $Z(Q)\cap C_{s-1}$ is normal in $G$ and
so by the minimality of $C_{s-1}$, we have that $C_{s-1}\leq Z(Q)$,
which implies that $C_t$ is normal in $Q$. Since $C_t$ is also
normalized by $P$ in $G_1$, we obtain that $C_t$ is normal in
$PQ=G$, a contradiction. Assume that $|C_{s-1}|=q^b$. Applying Lemma
2.4, we have that $b$ is the exponent of $q (mod \ p^{\alpha})$.
Considering the factor group $G/C_{s-1}$, we observe that  the chief
series of $G/C_{s-1}$ is
$$G/C_{s-1}>\cdots >C_{\alpha-1}/C_{s-1}>Q/C_{s-1}>C_{\alpha+1}/C_{s-1}>\cdots >C_{s-2}/C_{s-1}>1.$$
By induction, we obtain that the chief series of $G$ has the the
following type:
$$\underbrace{p, \ldots, p}_{\alpha}; \underbrace{q^{b}, \ldots, q^b}_{k}, \ \ \beta=kb,$$
where $b$ is the exponent of $q (mod \ p^{\alpha})$. Since $C_{s-1}\leq
Z(Q)$, the analogous argument induces that $C_{s-2}/C_{s-1}\leq
Z(G/C_{s-1})$. By induction, we conclude that $C_{i}/C_{i+1}$ is
contained in $Z(Q/C_{i+1})$ with $i=\alpha+1, \ldots, s-1$. It
follows that $Q>C_{\alpha+1}>\cdots >C_{s-1}>1$ is a central series
of $Q$ with length of $k$. Therefore the nilpotency class is bounded
by $k$.

Now, we may assume that $Q$ is not a Sylow $q$-subgroup of $G$. Then
$G/Q$ is a metacyclic group of order $p^{\alpha}q^{\gamma}$. Let
$H/Q$ be a Sylow $p$-subgroup of $G/Q$. Then $H/Q$ is normal in
$G/Q$. By Corollary 1.1, we have that $q^{\gamma}|(p-1)$ and $G$ has
a normal series $$G>H>Q>1,$$where the corresponding indices are
$q^{\gamma}$, $p^{\alpha}$, $q^{\beta-\gamma}$, respectively. We can
refine above normal series as $$G>\cdots >H>\cdots >Q>C_1>\cdots
C_k>1,$$where the orders of each chief factor is as
$$\underbrace{q, \ldots, q}_{\gamma}; \underbrace{p, \ldots, p}_{\alpha}; q^{b_1}, \ldots, q^{b_k}.$$
At last, we consider the following normal series of $H$:
$$H>\cdots >Q>C_1>\cdots C_k>1.$$
Observing that  $Q$ is a normal Sylow $q$-subgroup of $H$, we see
that the orders of chief factor of $H$ are $$p, \ldots, p; q^b,
\ldots, q^b.$$ Refining the chief series of $H$ above, we get that
$b|b_i$ for $1\leq i\leq k$.

Since $q^{\gamma}|(p-1)$ and $p^{\alpha}|(q^b-1)$, we have that
$$q^\gamma\leq p-1\leq p^{\alpha}-1<p^{\alpha}\leq q^b-1<q^b,$$and
so $\gamma<b$. (??)

(2) Let $P$ be a Sylow $2$-subgroup of $G$ and suppose that $P$ is a
generalized quaternion group. It follows from Theorem 2.3 that $G$
has a normal Sylow $q$-subgroup $Q$. Since $P$ has a cyclic subgroup
$K$ of order $2^{\alpha-1}$, we have that $H=KQ$ is normal in $G$,
and $$G>H>Q>1$$ is a normal series of $G$. Using a similar argument
as in (1), we conclude that $G$ has a chief series with every factor
having order as $$\underbrace{2, \ldots, 2}_{\alpha}; q^{b_1},
\ldots, q^{b_k},$$ where $b|b_i$ for $i=1, \ldots, k$ and $b$ is the exponent of $q (mod \ 2^{\alpha-1})$. Now we prove that each $b_i>1$ and
it suffices to prove $b_k>1$ by induction. Let $C_K$ be a minimal
normal subgroup of $G$ and assume that $|C_K|=q$. Write $L=PC_K$.
Since $L$ is an EPPO-group, we have that $C_K$ is centralized by
itself in $L$ and thus $L/C_K\simeq P$ is isomorphic to a subgroup
of $\mathrm{Aut}(C_K)$, contradicting that $\mathrm{Aut}(C_K)$ is
cyclic. Hence $b_k>1$ and the result follows.
\end{proof}

Notice that Theorem 2.4 is a refinement of Theorem 1 in \cite{Hig}.

\noindent {\bf Corollary 2.1.} \emph{Let $G$ be an EPPO-group. Then
$G$ is supersolvable if and only if $G$ has a normal subgroup of
order $q$ with $q\in \pi(G)$.}

\begin{proof}
It suffices to prove the sufficiency part.

Assume first that $q$ is an odd prime. By (2) in Theorem 2.3, $G$ is
a solvable group of order $p^{\alpha}q^{\beta}$. By Corollary 1.1,
$p^{\alpha}$ divides $q-1$, then the exponent of $q (mod \ p^{\alpha})$ is 1, and so the Sylow $p$-subgroups of $G$ are
cyclic. Let $|G/Q|=p^{\alpha}q^{\gamma}$. Then, by Theorem 2.4, we
have $\gamma=0$, which indicates that $Q$ is a normal Sylow
$q$-subgroup of $G$. Since $p^{\alpha}|(q-1)$, by Theorem 2.4 again,
$G$ has a chief series such that every factor has order as the
following:$$\underbrace{p, \ldots, p}_{\alpha}; \underbrace{q,
\ldots, q}_{\beta} \ ,$$implying that $G$ is supersolvable, as
wanted.

If $q=2$, then $G$ has a central element of order $2$ and it follows
that $G$ is a $2$-group by the hypothesis. Thus, the result is
clear.
\end{proof}

\section{Non-solvable EPPO-groups}

From the result of M. Suzuki (\cite{Su} Part 3, Theorem 5), we need only discuss the $ZT$-groups,  $G=PSL_2(q)$ with $q$ a Fermat primes or a Mersenne primes, $q=4, 9$, $PSL_3(4)$ and $M_9$.

\noindent {\bf Lemma 3.1.} \emph{A $ZT$-group $G$ is an EPPO-group
if and only if $G$ is isomorphic to $PSL_2(2^2)$, $PSL_2(2^3)$,
$Sz(2^3)$ or $Sz(2^5)$.}

\noindent {\bf Lemma 3.2.} \emph{Let $G=PSL_2(q)$ with $q$ a Fermat
primes or a Mersenne primes. Then $G$ is an EPPO-group if and only if
$q=5, 7, 17$.}

\noindent {\bf Lemma 3.3.} \emph{$PSL_2(9)$ and $M_9$ are
EPPO-groups.}

\noindent {\bf Lemma 3.4.} \emph{$PSL_3(4)$ is an EPPO-group.}

\noindent{\bf Remark 2.} Through concrete calculations we get Lemma 3.1 to Lemma3.4. Notice that the orders of cyclic subgroups of the above groups all are prime powers.

\noindent {\bf Theorem  3.1.} \emph{Suppose that $G$ is a
non-solvable EPPO-group. Then one of the following holds.}

(1) \emph{$G$ is isomorphic to $PSL_2(q)$ for $q=7, 9$ or $PSL_3(4)$
or $M_9$.}

(2) \emph{There exists a normal $2$-subgroup $T$ such that $G/T$ is
isomorphic to}

(i) \emph{$PSL_2(q)$ for $q=5, 8, 17$. In this case, the class
length of $T$ is not greater than 2 and $|T|$-1 is divisible by
$3\cdot 5$, $3^2\cdot 7$, $3^2\cdot 17$ respectively.}

(ii) \emph{$Sz(2^3)$ with $5\cdot 7\cdot 13|(|T|-1)$.}

(iii) \emph{$Sz(2^5)$ with $5^2\cdot 31\cdot 41|(|T|-1)$.}

\noindent {\bf Theorem  3.2.} \emph{Let $G$ be a non-abelian simple
EPPO-group. Then $G$ is one of the following groups:
$$A_5, PSL_2(7), PSL_2(8), PSL_2(9), PSL_2(17), PSL_3(4), Sz(2^3), Sz(2^5).$$}

\noindent {\bf Theorem  3.3.} \emph{Let $G$ be a group. Then
$G\simeq A_5$ if and only if there are at least 3 different primes
in $\pi(G)$ and the order of each non-trivial element in $G$ is a
prime.}

\noindent{\bf Remark 3.}  Let $G$ be a finite group. Then $G\simeq A_5$ if and only if $\pi_e(G) = \{1, 2, 3, 5\}$, where $\pi_e(G)$ denote the set of element orders of $G$.

\medskip

Acknowledgements. The authors are thankful to Dr. Jinbao Li for
reading this Chinese paper and translate it into English.







\end{document}